\documentclass[12pt]{amsart}

 \usepackage{amsfonts}
\usepackage{amsmath}
\usepackage{amsthm}
\usepackage{mathrsfs}
\usepackage[v2]{xy}
\usepackage{amscd}  
\usepackage{amssymb}
\usepackage{graphicx}
\usepackage{enumerate}
\usepackage{color}
\usepackage{setspace}    
\usepackage{array}
   
\newtheorem{theorem}{Theorem}[section]
\newtheorem{corollary}[theorem]{Corollary} 
\newtheorem{lemma}[theorem]{Lemma}
\newtheorem{proposition}[theorem]{Proposition} 
 
\theoremstyle{definition}
\newtheorem{definition}[theorem]{Definition}

\theoremstyle{example}
\newtheorem{example}[theorem]{Example}

\theoremstyle{remark}

\numberwithin{equation}{section}

\newcommand{\hs}{\mathcal{H}}

\def\C{{\mathbb C}}
\def\R{{\mathbb R}}
 
\def\Z{{\mathbb Z}}
\def\N{{\mathbb N}}
\def\T{{\mathbb T}}
\def\wmu{\widehat{\mu}}

\begin{document}
\allowdisplaybreaks[4]
  
\title[Scalar spectral measures]{Scalar spectral measures associated with an Operator-Fractal}
\author[P.E.T. Jorgensen]{Palle E. T. Jorgensen}
\address[Palle E.T. Jorgensen]{Department of Mathematics, The University of Iowa, Iowa
City, IA 52242-1419, U.S.A.}
\email{jorgen@math.uiowa.edu}
\urladdr{http://www.math.uiowa.edu/\symbol{126}jorgen/}
\author[K.A. Kornelson]{Keri A. Kornelson}
\address[Keri Kornelson]{Department of Mathematics, The University of Oklahoma, Norman, OK, 73019-0315, U.S.A.}
\email{kkornelson@math.ou.edu}
\urladdr{http://www.math.ou.edu/\symbol{126}kkornelson/}
\author[K.L. Shuman]{Karen L. Shuman}
\address[Karen Shuman]{Department of Mathematics and Statistics,
Grinnell College, Grinnell, IA 50112-1690, U.S.A.}
\email{shumank@math.grinnell.edu}
\urladdr{http://www.math.grinnell.edu/\symbol{126}shumank/}
\thanks{The second and third authors were supported in part by NSF grant DMS-0701164.  The third author was supported in part by the Grinnell College Committee for the Support of Faculty Scholarship.}

\let\thefootnote\relax


\begin{abstract} 
We examine the operator $U_5$ defined on $L^2(\mu_{\frac14})$ where $\mu_{\frac14}$ is the $\frac14$ Cantor measure.   The operator $U_5$ scales the elements of the canonical exponential spectrum for $L^2(\mu_{\frac14})$ by $5$ --- that is, $Ue_{\gamma} = e_{5\gamma}$ where $e_{\gamma}(t) = e^{2\pi i \gamma t}$.  It is known that $U_5$ has a self-similar structure, which makes its spectrum, which is currently unknown, of particular interest.  In order to better understand the spectrum of $U_5$, we demonstrate a decomposition of the projection valued measures and scalar spectral measures associated with $U_5$.  We are also able to compute associated Radon-Nikodym derivatives between the scalar measures.  Our decomposition utilizes a system of operators which form a representation of the Cuntz algebra $\mathcal{O}_2$.
\end{abstract}

\date{\today}
\maketitle

\tableofcontents

%\begin{doublespace}
\section{Introduction}\label{Sec:Intro}

\subsection{Background and setting}
Over a decade ago, it was discovered that there exists a family of Borel measures $\{\mu\}$ with compact support on $\R$ such that each $\mu$ has fractal dimension and $L^2(\mu)$ possesses a Fourier basis $\{e^{2\pi i \gamma t} : \gamma \in \Gamma_{\mu}\}$ \cite{JoPe98}.  In this setting, a Fourier basis for $L^2(\mu)$ is determined by a countable set $\Gamma_{\mu}$, which is called the \textbf{spectrum} of $\mu$.  If $\mu$ has such a Fourier basis, then $\mu$ is called a \textbf{spectral measure}.  Saying that $L^2(\mu)$ possesses a Fourier basis is the same as saying $L^2(\mu)$ admits orthogonal Fourier series.

Not all measures are spectral measures.  For example, the middle-thirds Cantor measure $\mu_{\frac13}$ does not have a spectrum---in fact, there cannot be more than two orthogonal complex exponentials in $L^2(\mu)$. On the other hand, many Cantor measures are spectral \cite{JoPe98}.  If $\mu$ is determined by scaling by $\frac{1}{4}$ at each Cantor iteration step, then the corresponding space $L^2(\mu_{\frac14})$ does have a Fourier basis.   We focus on $L^2(\mu_{\frac14})$ in this paper.

Typically, a spectrum for $\mu$ is a relatively ``thin'' subset of $\Z$ or $\frac{1}{2}\Z$ which has its own scaling properties related to the scaling invariance of $\mu$.  Sometimes a spectrum displays invariance with respect to two different scales, and in these cases, many questions arise.   A particularly interesting example is the Jorgensen-Pedersen spectrum $\Gamma(\frac{1}{4})$, which has self-similarity when scaled by $4$; in addition, the set $5\Gamma(\frac{1}{4})$ is also a spectrum for $\mu_{\frac14}$.  In this paper, we continue to study the two scaling operations, scaling by $4$ and scaling by $5$, whose intertwining properties were first discovered in \cite{DJ09} and later considered in \cite{JKS11} and \cite{JKS12}. 

Despite a number of investigations into the spectral properties of affine measures, there are still many open questions.  For example, even for the simplest case of $\mu_{\frac{1}{4}}$, scaling by $5$ is not well-understood.  Scaling by $4$ induces a $\mu_{\frac{1}{4}}$-measure preserving transformation.  In \cite{JKS12}, we showed that  scaling by $5$ does \textbf{not} preserve $\mu_{\frac{1}{4}}$. Nonetheless, by the result in \cite{DJ09}, scaling $\Gamma(\frac{1}{4})$ by $5$ induces a unitary operator $U$ in  $L^2(\mu_{\frac{1}{4}})$, so the two scalings, one by $4$ and the other by $5$, are compatible at the level of operators. 

Our paper is devoted to studying the $5$-scaling operator $U$.  Its spectral theory is surprisingly subtle.   While $U$ is induced by scaling a spectrum for $\mu_{\frac{1}{4}}$ by $5$, $U$ is not the lifting of a $\mu_{\frac{1}{4}}$-measure preserving endomorphism.  But the operator $U$ has a ``fractal'' nature of its own---it is the  countable infinite direct sum of the operator $MU$ plus a rank-one projection.  Here,  $M$ is multiplication by $z$ in a Fourier representation of $L^2(\mu_{\frac{1}{4}})$.

We aim to study the spectral theory of $U$ with regard to the geometry and ergodic theory of the initial spectral pair $(\mu_{\frac14}, \Gamma(\frac14))$ (Equation \eqref{Eqn:Gamma}) with the use of a natural representation of the Cuntz algebra $\mathcal{O}_2$ acting on the Hilbert space $L^2(\mu_{\frac{1}{4}})$;  in our analysis we make further use of reduction to cyclic subspaces in  $L^2(\mu_{\frac{1}{4}})$.

In a main result for $U$ \cite[Theorem 4.10]{JKS11}, we proved that $U$ is an orthogonal sum of a one-dimensional projection and an infinite number of copies of the operator $MU$.  In other words, $U$ is an ``operator fractal''; i.e., it is a geometric representation of an infinite number of scaled versions of itself. By ``orthogonal sum'' we mean that $L^2( \mu_{\frac14})$ is an orthogonal sum of closed invariant subspaces for $U$.

Now, by the spectral theorem applied to $U$, we obtain an associated projection valued measure $P^U$ supported on the Borel subsets of the circle group $\T$. Hence, from $P^U$ we get induced scalar spectral measures, each one induced by cyclic subspaces for $U$, and hence indexed by vectors  $v \in L^2( \mu_{\frac14})$. For each $v$, we get an associated scalar measure  $m^U_v$ (see Equation \eqref{Eqn:m^{X}_v}), and the spectral data for $U$ is carried by this family of measures.

Below, in Proposition \ref{Prop:RealDecomp} and Theorem \ref{Thm:Iterate} we write out formulas for each of the scalar measures  $m^U_v$  which turn the notion of ``operator fractal'' into a more precise spectral theoretic theorem.   Proposition \ref{Prop:RealDecomp} reflects the $\mathcal{O}_2$ splitting of $L^2(\mu_{\frac14})$ and Theorem \ref{Thm:Iterate} reflects the fractal nature of $U$, in Equation \eqref{Eqn:U_Decomp} and the subsequent matrix decomposition.   Specifically, we prove that, for every $v$,  $m^U_v$  is a convex sum of scalar times the Dirac mass at $1$ and an infinite number of copies measures computed from the operator $MU$.

The paper is organized as follows:  our main decomposition theorems are in Sections \ref{Sec:Decomp} and \ref{Sec:Cyclic}---Theorem \ref{Thm:Iterate}, Theorem \ref{Thm:NoProj}, and Corollary \ref{Cor:Convex}. But in the remaining of this section and the next, we prepare the ground with the statement and proof of some key lemmas to be used.  We hope that these preliminary results may also be of independent interest.  In particular, our new measures in Section \ref{Sec:Induced} play a role in our main Theorem \ref{Thm:Iterate}.  Our preliminary results deal with representation of certain algebras on generators and relations (Theorem \ref{Thm:GenRel}), as well as some results from operator theory and from the theory of measures on boundary spaces.  Inside algebras on generators and relations we identify a particular representation (Lemmas \ref{Lemma:Extend} and \ref{Lemma:Adjoint}) of the Cuntz algebra $\mathcal{O}_2$, i.e., the Cuntz algebra with two generators. This representation will play a key role in the rest of the paper.

Since \cite{JoPe98}, there has been a substantial literature devoted to the study of ergodic scaling properties of spectral affine measures.  A small sampling related directly to this paper includes \cite{DJ09},\cite{DuJo09a}, \cite{DuJo09c}, \cite{DHS09}, \cite{JKS11b}, and \cite{JKS12}.  For a wider look at the spectral theory for affine dynamical systems over the last fifteen years, we point to the papers \cite{JoPe98}, \cite{Gab00}, \cite{PeWa01}, \cite{LaWa02}, \cite{Pe04a}, \cite{Pe04b}, \cite{FeWa05}, \cite{LaWa06}, \cite{Str06}, \cite{HL08}, \cite{DHS09}, \cite{Li09}, \cite{BoKe10}, \cite{DHSW11}, and \cite{Li11}.

\subsection{Generators and relations:  An $L^2(\mu_{\frac{1}{4}})$ representation}\label{Sec:GenRel}

Axiomatic settings for the results in this paper can be framed in several different ways.  Here, we consider $L^2(\mu_{\frac{1}{4}})$ and our operators in the following context:

\begin{equation}\label{Ax:1}
\textrm{Hilbert space, operators:  }\begin{cases} 
& \hs \textrm{ a Hilbert space }\\
& \{S_i\}_{i = 0}^1 \in \textrm{Rep}(\mathcal{O}_2, \hs)\\
& U \textrm{ a normal operator on }\hs\\
& M \textrm{ a unitary operator on }\hs.
\end{cases}
\end{equation}
and
\begin{equation}\label{Ax:2}
\textrm{Operator relations:  }
\begin{cases} 
& S_1^* U S_1 = MU\\
& S_0 \textrm{ commutes with } U.\\
\end{cases}
\end{equation}
The two operators $\{S_0, S_1\}$ which form a representation of $\mathcal{O}_2$ on $\hs$ satisfy $S_i^*S_j = \delta_{i,j}I$, $i, j \in\{0, 1\}$ and $\sum_{i=0}^1 S_iS_i^* = I$.

\begin{theorem}\label{Thm:GenRel}
The relations \eqref{Ax:1} and \eqref{Ax:2} have an irreducible representation on $L^2(\mu_{\frac14})$.
\end{theorem}

We postpone the proof of the theorem until later in this section.

Our specific application of \eqref{Ax:1} and $\eqref{Ax:2}$ is to the Hilbert space $L^2(\mu_{\frac{1}{4}})$.   The Hilbert space $L^2(\mu_{\frac{1}{4}})$ has an associated spectrum, $\Gamma(\frac{1}{4})$, defined by
\begin{equation}\label{Eqn:Gamma}
 \Gamma\Bigl(\frac{1}{4}\Bigr) = 
\Biggl\{ 
\sum_{i=0}^m a_i 4^i \: : \: a_i\in \Bigl\{0,1 \Bigr\}
\Biggr\} = \{ 0, 1, 4, 5, 16, 17, 20, \ldots \}
\end{equation}
which in turn gives rise to an orthonormal basis (ONB) for $L^2(\mu_{\frac{1}{4}})$:
\begin{equation}
E\Bigl(\Gamma\Bigl(\frac{1}{4}\Bigr)\Bigr) = \Bigl\{  e_{\gamma}(t) = e^{2\pi i \gamma t} \: | \: \gamma \in \Gamma\Bigl(\frac{1}{4}\Bigr) \Bigr\}
\end{equation}
\cite[Corollary 5.9]{JoPe98}.  By  \cite[Proposition 5.1]{DJ09}, the scaled set $5\Gamma(\frac{1}{4})$ is also a spectrum for $\mu_{\frac14}$, which leads us to define the unitary operator $U$ by
\begin{equation}\label{Defn:U}
Ue_{\gamma} = e_{5\gamma}.
\end{equation}
The operator $M = M_{e_1}$ is multiplication by the exponential $e_1$:
\begin{equation}\label{Defn:M}
M_{e_1}e_{\gamma} = e_{\gamma + 1}.
\end{equation}

In \cite[Theorem 4.8]{JKS11}, the authors showed that the operator $U$ has a fractal-like nature which arises from a representation of $\mathcal{O}_2$ on $L^2(\mu_{\frac14})$ given by the operators
\begin{equation}\label{Eqn:Cuntz}
S_0 e_{\gamma} = e_{4\gamma} \textrm{ and } S_1 e_{\gamma} = e_{4\gamma+1}.
\end{equation}
The Cuntz operators $S_0$ and $S_1$ give rise to an ordering of the basis $E(\Gamma(\frac{1}{4}))$ and a resulting orthogonal decomposition of $L^2(\mu_{\frac{1}{4}})$ given by
\begin{equation}\label{Eqn:Decomp}
L^2(\mu_{\frac{1}{4}}) = \textrm{span}\{e_0\} \oplus \bigoplus_{k=0}^{\infty} S_0^kS_1L^2(\mu_{\frac14}).
\end{equation}
The subspaces $S_0^kS_1L^2(\mu_{\frac14})$ have the property that the matrix of $U$ restricted to each subspace is the same.  

The second set of axioms \eqref{Ax:2} is satifsied by $U$, $M_{e_1}$, $S_0$ and $S_1$.  We have
\begin{equation}\label{Eqn:Rel}
S_0U = US_0  \textrm{ and } S_1^* U S_1 = M_{e_1}U
\end{equation}
by \cite[Theorem 4.3]{JKS11} and \cite[Theorem 4.10]{JKS11}, respectively.

The operator $U$ can be written
\begin{equation}\label{Eqn:U_Decomp} 
U =  P_{e_0} \oplus \bigoplus_{k=0}^{\infty} M_{e_1}U,
\end{equation} 
where $P_{e_0}$ is the orthogonal projection onto $\textrm{span}\{e_0\}$; $P_{e_0}$ is the projection onto the unitary part of the Wold decomposition of $S_0$ (see Section \ref{Subsec:Wold} below).

We now prove Theorem \ref{Thm:GenRel}.

\noindent\textit{Proof of Theorem \ref{Thm:GenRel}}.  
Consider the $*$-algebra $\frak{A}$ generated by $U$, $M_{e_1}$, and the representation of $\mathcal{O}_2$ in $L^2(\mu_{\frac14})$.  Note that the representation of $\mathcal{O}_2$ carries its own relations, but as of now, with an abuse of notation, we have specified no relations among $U$, $M_{e_1}$, and the representation of $\mathcal{O}_2$.

Let $\frak{I}$ be the two-sided ideal generated by the relations which have already been established in \eqref{Eqn:Rel}:
\begin{equation}
S_0U - US_0 = 0 \textrm { and } S_1^*US_1 - M_{e_1}U = 0.
\end{equation}
We want to establish that $S_0$ and $S_1$ are not in $\frak{I}$.  But neither can be in $\frak{I}$ because if for $i = 0$ or  $i =1$
\[ S_i = c_1 X_1 (S_0U - US_0) Y_1 + c_2 X_2 (S_1^*US_1 - M_{e_1}U) Y_2\]
for $X_1, X_2, Y_1, Y_2\in \frak{A}$, then we could multiply both sides by $S_i^*$, which tells us that $I\in\frak{I}$, or that $S_i = 0$ in the $*$-algebra $\frak{A}/\frak{I}$, which is not true.

We explicitly establish that the representation of $\mathcal{O}_2$ is irreducible in \cite{JKS12b} so the representation of $\frak{A}/\frak{I}$ is also irreducible.
\hfill$\Box$

The matrix of $U$ with respect to the decomposition we have just described is given by
\begin{equation*}\label{Eqn:MatrixU_03}
\begin{tabular}{c||c|c|c|c|c|c}
                         & \small$\textrm{span}\{e_0\}$ & $S_1$  & $S_0S_1$ & $S_0^2S_1$ & $S_0^3S_1$ & $\cdots$\\
\hline
\hline
$\small\textrm{span}\{e_0\}$ & $1$ & $0$ & $0$ & $0$ & $0$& $\cdots$\\ 
\hline $S_1$  &$0$ & $M_{e_1}U$ &    $0$          &       $0$    &     $0$    &    $\cdots$ \\
\hline 
$S_0S_1$  & $0$ &  $0$           &$M_{e_1}U$ &       $0$    &     $0$    &    $\cdots$ \\
\hline
$S_0^2S_1$  & $0$ &  $0$           &   $0$            &        $M_{e_1}U$   &     $0$    &    $\cdots$ \\
\hline
$S_0^3S_1$  & $0$ &  $0$           &   $0$            &        $0$                &    $M_{e_1}U$  &    $\cdots$ \\
\hline
                        & \vdots &  $\vdots$   &    $\vdots$   &      $\vdots$           &   $\vdots$         &  $\ddots$ \\ 
\end{tabular} \quad
\end{equation*}
\cite[Theorem 4.10]{JKS11}.

To understand the operator $U$ better, we wish to compute its spectrum.  Currently, the spectrum of the operator $U$ is unknown, and the results in this paper are inroads to understanding the spectrum of $U$.

\section{The Cuntz operators and $U$}\label{Sec:Cuntz}

In our decomposition theorems we will employ a certain representation of an algebra on generators and relations.  Below, we prove that the representation has a number of properties to be used later.  Results from operator theory are in Section \ref{Subsec:Wold}; in Section \ref{Sec:Induced} we work with the theory of measures on boundary spaces. Inside algebras on generators and relations we identify a particular representation (Lemma \ref{Lemma:Adjoint}) of the Cuntz algebra $\mathcal{O}_2$, the Cuntz algebra with two generators. 
\subsection{Properties of $S_0$ and $S_1$}
The operators $S_0$ and $S_1$ defined on $E(\Gamma(\frac{1}{4}))$ in Equation \eqref{Eqn:Cuntz}
satisfy the Cuntz relations by \cite[Proposition 3.2]{JKS11}.   In fact, 
\begin{lemma}\label{Lemma:Extend}
The formulas for $S_0$ and $S_1$ extend to any $e_n$ where $n\in\Z$: 
\[S_0 e_{n} = e_{4n} \textrm{ and } S_1 e_{n} = e_{4n+1}.\]
\end{lemma}
\begin{proof}
The straightforward computations rely on the parity of $n$, the fact $\wmu(\textrm{odd integer}) = 0$, the scaling invariance $\wmu(4m) = \wmu(m)$ for each $m\in\Z$, the decomposition $\Gamma(\frac{1}{4}) =  4\Gamma(\frac{1}{4}) \sqcup (4\Gamma(\frac{1}{4}) + 1)$, and the containment $\Gamma(\frac{1}{4})\subset \Z$.\end{proof}
The adjoints of $S_0$ and $S_1$ on the basis $E(\Gamma(\frac14))$ are given by
\begin{equation}\label{Eqn:S0Adj}
S_0^* e_{\gamma} = \begin{cases} e_{\frac{\gamma}{4}} & \gamma \equiv 0 \pmod{4} \\
                                                                                          0  & \gamma \equiv 1 \pmod{4}\end{cases}
\end{equation}
and
\begin{equation}\label{Eqn:S1Adj}
 S_1^* e_{\gamma} = \begin{cases} e_{\frac{\gamma -1}{4}} & \gamma \equiv 1 \pmod{4} \\
                                                                                          0  & \gamma \equiv 0 \pmod{4}.\end{cases}
\end{equation}
 \cite[Equations (3.3) and (3.4)]{JKS11}.
When we ask if similar equations hold for $e_n$, $n\in\Z$, we find that three times out of four, the expected equations hold:
\begin{lemma}\label{Lemma:Adjoint}The adjoints of $S_0$ and $S_1$ satisfy the following equations:
\begin{itemize}
\item If $n\equiv 0 \pmod{4}$, then $S_0^*e_n = e_{\frac{n}{4}}$.  
\item If $n\equiv 1,3 \pmod{4}$, then $S_0^* e_n = 0$. 
\item If $n\equiv 1 \pmod{4}$, then $S_1^*e_n = e_{\frac{n-1}{4}}$.  
\item If $n\equiv 0,2 \pmod{4}$, then $S_1^* e_n = 0$. 
\end{itemize}
\end{lemma}
\begin{example}When $n\equiv 2 \pmod 4$
\end{example}
We note that if $n\equiv 2\pmod{4}$, then there is no clean formula for $S_0^*e_n$.  For example, when $n = 2$, we have 
\begin{equation*}
\begin{split}
S_0^*e_2 
& = S_0^*\sum_{\gamma\in \Gamma} \wmu(\gamma - 2) e_{\gamma}\\
& = S_0^*\sum_{\gamma\in \Gamma} \wmu(4\gamma - 2) e_{4\gamma} + S_0^*\sum_{\gamma\in \Gamma} \wmu(4\gamma + 1 - 2) e_{4\gamma+1}\\
& = \sum_{\gamma\in \Gamma} \wmu(4\gamma - 2) e_{\gamma},\\
\end{split}
\end{equation*}
and \textit{none} of the coefficients $\wmu(4\gamma - 2)$ is zero, since $4\gamma - 2$ can never be written in the form $4^k(2n + 1)$ for any $k\in \N\cup \{0\}$, $n\in \Z$.
\hfill$\Diamond$

\begin{definition}The \textbf{commutant} of $U$, denoted $\{U\}'$, is the set of all bounded operators which commute with both $U$ and $U^*$.
\end{definition}
\begin{lemma}\label{Lemma:S0UUS0}
The operator $S_0$ belongs to $\{U\}'$.
\end{lemma}
\begin{proof}
The operators $U$ and $S_0$ commute \cite[Theorem 4.3]{JKS11}.   Since $U$ is unitary, $U$ is normal, so $U$ commutes with $S_0^*$ by Fuglede's theorem \cite[Theorem 1, p.~35]{Fug50}.  
\end{proof}

\noindent\textbf{Remark:  }It is also possible to check directly on the ONB $\{e_{\gamma} : \gamma\in \Gamma(\frac{1}{4})\}$ that
\[ \langle U^*S_0 e_{\gamma}, e_{\xi} \rangle = \langle S_0U^* e_{\gamma}, e_{\xi}\rangle \]
for all $\gamma, \xi\in\Gamma(\frac{1}{4})$.

On the other hand, $U$ and $S_1$ do not commute, although $U$ and the range projection $S_1S_1^*$ do commute.  To see this, use the fact that $U$ commutes with $S_0$ and $S_0^*$ and $S_0S_0^* +  S_1S_1^* = I$.

\subsection{The Wold decomposition of $S_0$}\label{Subsec:Wold}

The general theory of Wold decompositions applies to an isometry $A$ on a Hilbert space $\hs$.  The Wold decomposition of $A$ on $\hs$ is an orthogonal decomposition of $\hs = \hs_{s} \oplus \hs_{u}$ such that
\begin{enumerate}[(i)]
\item $A|_{\hs_{u}}$ is unitary.
\item $A|_{\hs_{s}}$ is a shift.
\item Both $\hs_{s}$ and $\hs_{u}$ are invariant under $A$
\end{enumerate}
\cite[p.~96]{BJ2002} or \cite{SzNagyFoias}.

In addition, the unitary part $\hs_{u}$ of the Wold decomposition of $A$ has two other properties:
\begin{proposition}\label{Prop:H_u}\rm\cite[p.~96]{BJ2002} \it Suppose $A$ is an isometry on $\hs$ with Wold decomposition $\hs=\hs_{u}\oplus \hs_{s}$.  Then
\begin{enumerate}[\rm(1)]
\item $A^n(A^n)^*\rightarrow P_{\hs_{u}}$ as $n\rightarrow\infty$, where $P_{\hs_{u}}$ is the projection onto $\hs_{u}$.  (The convergence is in the strong operator topology and is in fact monotonic.)
\item $\hs_{u} = \{ f\in \hs \:  : \: \|(A^n)^*f\| = \|f\| \textrm{ for all }n\in\N\}$.
\end{enumerate}
\end{proposition}

\begin{proposition}\label{Prop:WoldS_0}
Let $S_0$ be the operator defined on $L^2(\mu_{\frac14})$ in Equation \eqref{Eqn:Cuntz}:  $S_0 e_{\gamma} = e_{4\gamma}$ for $\gamma\in \Gamma(\frac14)$.  Then the following two equivalent conditions hold:
\begin{enumerate}[\rm(a)]
\item $\displaystyle\lim_{n\rightarrow\infty} S_0^n(S_0^n)^*f = \langle e_0, f\rangle e_0$.
\item $\| (S_0^n)^*f \| = \|f\|$ for all $n\in\N$ if and only if $f = ce_0$ for some $c\in\C$.
\end{enumerate}
\end{proposition}

\begin{proof}
We will use the general Wold decomposition's characterization of $\hs_{u}$ in Proposition \ref{Prop:H_u}, part $(2)$ to show that for $A = S_0$, the unitary space $\hs_{u} = \textrm{span}\{e_0\}$.  Suppose $f\in L^2(\mu)\ominus \textrm{span}\{e_0\}$.  By \cite[Proposition 3.4]{JKS11},
\[ L^2(\mu_{\frac14}) \ominus \textrm{span}\{e_0\} = \bigoplus_{k = 0}^{\infty} S_0^kS_1(L^2(\mu_{\frac14})).\]
Therefore we can write $f$ as
\[f = \sum_{k = 0}^{\infty} \sum_{\gamma\in\Gamma} c_{4^k(1+4\gamma)}S_0^kS_1e_{\gamma}.\]
Now, let $n\in\N$.  Because $S_0$ and $S_1$ satisfy the Cuntz relations, $S_0^*S_0$ is the identity, and $S_0^*S_1 = 0$.  As a result, for $k < n$, 
\[ (S_0^*)^nS_0^kS_1  = 0.\]
Therefore
\begin{equation}\label{Eqn:n-k_1}
(S_0^*)^nf = \bigoplus_{k = n}^{\infty} \sum_{\gamma\in\Gamma} c_{4^k(1+4\gamma)}S_0^{k-n}S_1e_{\gamma}.
\end{equation}
and
\begin{equation}\label{Eqn:n-k_2}
\|(S_0^*)^nf\|^2 = \sum_{k = n}^{\infty} \sum_{\gamma\in\Gamma} |c_{4^k(1+4\gamma)}|^2.
\end{equation}
Since Equation \eqref{Eqn:n-k_2} is true for every $n$ and 
\begin{equation}
\bigcap_{k = 0}^{\infty} 4^k(1+4\Gamma) = \emptyset,
\end{equation}
we have
\begin{equation}
\lim_{n\rightarrow\infty}\|(S_0^*)^nf\| = 0.
\end{equation}
So, the requirement that $\|f\| = \|(S_0^*)^nf\|$ for all $n\in\N$ forces
\begin{equation}
\hs_{u} \cap \Bigl(L^2(\mu_{\frac14}) \ominus \textrm{span}\{e_0\}\Bigr) = \{0\}.
\end{equation}

On the other hand, if $f\in \textrm{span}\{e_0\}$, then $f = ce_0$ for some $c\in\C$, and $\|f\| = \|(S_0^n)^*f\|$ for all $n\in\N$ by Equation \eqref{Eqn:S0Adj}.  Therefore $f\in \hs_{u}$.

Now (a) and (b) follow from Proposition \ref{Prop:H_u}.
\end{proof}

We can draw a more general conclusion from Proposition \ref{Prop:WoldS_0}.  When $L^2(\mu)$ has a ONB $\{e_{\gamma} : \gamma\in\Gamma\}$ where $\Gamma$ is a countable infinite set, there is an isometric isomorphism between $L^2(\mu)$ and $\ell^2(\Gamma)$: if $f(t) = \sum_{\gamma} c_{\gamma}e_{\gamma}(t)$ in an $L^2(\mu_{\frac14})$-expansion, then
\begin{equation}
\|f\|^2_{L^2(\mu)} = \sum_{\gamma} |c_{\gamma}|^2 = \|c\|^2_{\ell^2(\Gamma)}.
\end{equation}
Here we use duality (Parseval).  Using the same proof technique as in Proposition \ref{Prop:WoldS_0}, we can establish the following lemma.
\begin{lemma}
Let $X$ be a countable infinite set, and let $\tau:X\rightarrow X$ be an injective but not surjective endomorphism such that
\begin{equation}
\bigcap_{n=1}^{\infty}\tau^n(X) = \{x_0\}
\end{equation}
where $\{x_0\}$ is a singleton.
Let $\xi = \xi(x)\in L^2(X)$, and let $S:\ell^2(X) \rightarrow \ell^2(X)$ and be defined by
\begin{equation}
(S\xi)(x) = \xi(\tau(x)).
\end{equation}
Then the unitary part of the Wold decompsition of $S:\ell^2(X)\rightarrow \ell^2(X)$ is one-dimensional; that is,
\begin{equation}
[\ell^2(X)]_u = \C\delta_{x_0}.
\end{equation}
\end{lemma}

\subsection{Induced measures on infinite words}\label{Sec:Induced}
Let $v\in L^2(\mu_{\frac14})$, and expand $v$ in terms of the ONB $E(\Gamma(\frac14))$:
\begin{equation}\label{Eqn:v}
v = \sum_{\gamma\in\Gamma(\frac14)} c_v(\gamma) e_{\gamma},
\end{equation}
in an $L^2(\mu_{\frac14})$-expansion, where
\begin{equation}\label{Eqn:Coef}
c_v(\gamma) = \langle e_{\gamma}, v\rangle_{L^2(\mu_{\frac14})}.
\end{equation}
If
\begin{equation}
\|v\|^2_{L^2(\mu_{\frac14})} = \int |v(x)|^2 \:d\mu_{\frac14}(x) = 1,
\end{equation}
then Equation \eqref{Eqn:Coef} corresponds to a probability distribution.  Recall the definition of $\Gamma(\frac14)$ from Equation \eqref{Eqn:Gamma}:
\[
 \Gamma\Bigl(\frac{1}{4}\Bigr) = 
\Biggl\{ 
\sum_{i=0}^m a_i 4^i \: : \: a_i\in \Bigl\{0,1 \Bigr\}, m \textrm{ finite}
\Biggr\},
\]
which is in one-to-one correspondence with all the finite words in the bits $0$ and $1$.  In other words, $(a_0, a_1, \ldots, a_N)$ represents a word of finite length in $0$s and $1$s of length $N+1$; the bit $a_N$ is the last bit in the representation of the word.

Now consider
\begin{equation}
X = \prod_{0}^{\infty} \{0,1\},
\end{equation}
which can be viewed as a Cantor set in the Tychonoff compact topology \cite{Jor06}.  Then $X$ is the set of all \textit{infinite} words in the bits $0$ and $1$; an element $\omega\in X$ can be denoted $\omega = (\omega_0, \omega_1, \omega_2, \ldots)$.

Let $\gamma\in \Gamma(\frac14)$, and make the natural association $\gamma \leftrightarrow (a_0, a_1, \ldots, a_N)$, where $\gamma = \sum_{i=0}^N a_i 4^i$.  Let $A(\gamma)$ denote the corresponding cylinder set in $X$:
\begin{equation}
A(\gamma) = \{ \omega \in X : \omega_i = a_i,\: 0 \leq i \leq N\}.
\end{equation}

We introduce one more piece of notation:  we concatenate two words in $\Gamma(\frac14)$ in the natural way.  If $\gamma \leftrightarrow (a_0, \ldots, a_N)$ and $\xi \leftrightarrow (b_0, \ldots, b_K)$, then
\[ 
\gamma\xi \leftrightarrow (a_0, \ldots, a_N, b_0, \ldots, b_K)
\]
corresponds to 
\[ \sum_{i = 0}^N a_i 4^i + \sum_{i = 0}^K b_i 4^{1+N+i}\in \Gamma\Bigl(\frac14\Bigr).\]
If $v$ is given as in \eqref{Eqn:v} and \eqref{Eqn:Coef}, then define
\begin{equation}\label{Eqn:PGamma}
P^v_{\Gamma}(A(\gamma)) = \sum_{\xi\in\Gamma(\frac14)} |c_v(\gamma\xi)|^2.
\end{equation}

Recall the following properties of cylinder sets (see also \cite{Jor06}):
\begin{enumerate}[(i)]
\item If $\emptyset$ denotes the empty word, then we set $A(\emptyset) = X$.
\item If $\gamma, \gamma'\in\Gamma(\frac14)$, then
\[ A(\gamma)\cap A(\gamma') \neq \emptyset \Leftrightarrow \gamma = \gamma'.\]
\item If there exists $\xi$ such that $\gamma' = \gamma\xi$, then
\[ A(\gamma') \subset A(\gamma).\]
\end{enumerate}

\begin{lemma}\label{Lemma:Borel}
For every $v\in L^2(\mu_{\frac14})$ with $\|v\| = 1$, then the assignment $P_{\Gamma}^v(\cdot)$ defined in Equation \eqref{Eqn:PGamma} extends to a probability measure defined on the $\sigma$-algebra of all Borel subsets in $X$.
\end{lemma}
\begin{proof}See \cite{Jor06}.\end{proof}
\noindent\textbf{Remark:  }The role of the measures $P_{\Gamma}^v(\cdot)$ extends into our decomposition in Theorem \ref{Thm:Iterate} below.

\bigskip

\begin{proposition}
For every $v\in L^2(\mu_{\frac14})\backslash E(\Gamma(\frac14))$ with $\|v\| = 1$, the measure $P_{\Gamma}^v(\cdot)$ defined on $X$ is non-atomic---that is, $P_{\Gamma}^v(\{\omega\}) = 0$ for every $\omega \in X$.
\end{proposition}
\begin{proof}
Let $v\in L^2(\mu_{\frac14})\backslash E(\Gamma(\frac14))$ with $\|v\| = 1$, and let $\omega\in X$.  For every $n\in \Z^+$, let $\omega(n)$ be the level $n$ truncation of $\omega$:
\begin{equation}\label{Eqn:Trunc}
\omega(n) = (\omega_0, \omega_1, \ldots, \omega_n).
\end{equation}

As before, let $A(\omega(n))$ denote the corresponding cylinder sets.  We have
\begin{equation}
A(\omega(1)) \supset A(\omega(2)) \supset \cdots \supset A(\omega(n)) \supset \cdots,
\end{equation}
and
\begin{equation}
\bigcap_{n=1}^{\infty} A(\omega(n)) = \{\omega\}.
\end{equation}	

By Lemma \ref{Lemma:Borel}, $P_{\Gamma}^v(\cdot)$ is a Borel measure on $X$, so it follows that
\begin{equation}
 \lim_{n\rightarrow\infty} P^v_{\Gamma}(A(\omega(n))) =  P^v_{\Gamma}(\{\omega\}).
\end{equation}
We claim that the limit is $0$.  
Fix $n\in\N$.  By the definition in Equation \eqref{Eqn:PGamma},
\begin{equation}
P_{\Gamma}^v(A(\omega(n)) = \sum_{\xi\in\Gamma(\frac14)} |\,c_v(\omega(n)\xi)\,|^2.
\end{equation}
Now compare the relative roles of $\omega(n)$ and $\omega(n+1)$.  All the terms in the sum
\[\sum_{\xi\in\Gamma(\frac14)} |\,c_v(\omega(n+1)\xi)\,|^2\]
are contained in the sum
\[\sum_{\xi\in\Gamma(\frac14)} |\,c_v(\omega(n)\xi)\,|^2,\] 
(but not vice versa).  

Now let $N\gg n$.   The sum
\[\sum_{\xi\in\Gamma(\frac14)} |\,c_v(\omega(N)\xi)\,|^2\]
is in the ``tail'' of
 \[\sum_{\xi\in\Gamma(\frac14)} |\,c_v(\omega(n)\xi)\,|^2,\]
 and  $\{ c_v(\omega(n)\xi)\}_{\xi\in\Gamma}\in \ell^2(\Gamma(\frac14))$.  Therefore the limit as $n\rightarrow\infty$ is $0$.
\end{proof}

\begin{example}Some atomic measures.\end{example}
\noindent Consider the case $v = e_0$ and the element $\omega = $an infinite string of $0$s.      Then for any $n\in\N$,
\begin{equation}
\begin{split}
P_{\Gamma}^{e_0}(A(\omega(n))
& = \sum_{\xi\in\Gamma(\frac14)} |\,c_{e_0}(\omega(n)\xi)\,|^2\\
& =  |\,c_{e_0}(\omega(n)0)\,|^2 +  \sum_{\xi \in \Gamma(\frac14)\backslash\{0\}}  |\,c_{e_0}(\omega(n)\xi)\,|^2 \\
& = | \langle e_0, e_0\rangle|^2 +  \sum_{\xi \in \Gamma(\frac14)\backslash\{0\}}  |\langle e_0, e_{\omega(n)\xi}\rangle |^2=1.\\
\end{split}
\end{equation}

Denote by $\omega = \underline{0}$ the infinite string of $0$s in $X$.  Fix $\gamma\in\Gamma(\frac14)$.  Then the measure $P^{e_\gamma}_{\Gamma}$ is atomic, and its atom is given by $(\gamma\, \underline{0})\in X$.  The reasoning is the same as above, except we consider the cylinder sets corresponding to $(\gamma\, \underline{0})$, which we denote $A(\gamma\omega(n))$:
\begin{equation}
\begin{split}
P_{\Gamma}^{e_{\gamma}}(A(\gamma\omega(n)) 
 & =  |\,c_{e_{\gamma}}(\gamma\omega(n)0)\,|^2 +  \sum_{\xi \in \Gamma(\frac14)\backslash\{0\}}  |\,c_{e_{\gamma}}(\gamma\omega(n)\xi)\,|^2 \\
& = | \langle e_{\gamma}, e_{\gamma}\rangle|^2 +  \sum_{\xi \in \Gamma(\frac14)\backslash\{0\}}  |\langle e_{\gamma}, e_{\gamma\omega(n)\xi}\rangle |^2=1.\\
\end{split}
\end{equation}

\hfill $\Diamond$
\section{Measure decompositions}\label{Sec:Decomp}

Below we prove some lemmas on particular representations of $C^*$-algebras. They will be needed later.  In Section \ref{Sec:O_2}, we employ two $C^*$-homomorphisms to find a decomposition of an operator $X$ which commutes with both $S_0$ and $S_0^*$.  The resulting Equation \eqref{Eqn:Id1} tells us that any such $X$ can be written as
\[ X = S_0 X S_0^* +  S_1 S_1^* X S_1S_1^*.\] 
The operators $U$ and $\phi(U)$ ($\phi$ is a Borel function on $\T$) are two such operators which commute with $S_0$ and $S_0^*$, and our applications will be to these operators.
Then, in Section \ref{Subsec:pvms}, we decompose the projection-valued measure of $\phi(U)$ to derive Equation \eqref{Eqn:MeasDecomp}.  Finally, in Section \ref{Subsec:SpecMeas}, we explore the spectral real-valued measures associated with $U$ arising from our decomposition \eqref{Eqn:Id1}.  The resulting Proposition \ref{Prop:RealDecomp} allows us to write a real-valued spectral measure $m_v^U$ in terms of two pieces associated with $S_0$ and $S_1$. Finally, we provide several examples to show how the spectral real-valued measures can be computed.

\subsection{Two $C^*$-homomorphisms}\label{Sec:O_2}

Because the isometries $\{S_0, S_1\}$ form a representation of $\mathcal{O}_2$ on $L^2(\mu_{\frac{1}{4}})$, we know that the isometries $S_0$ and $S_1$ have orthogonal ranges and that the idenitity operator can be written as a sum of range projections:
\[
I = S_0S_0^* + S_1S_1^*.
\]

If $\hs$ is a Hilbert space, we denote by $B(\hs)$ the algebra of bounded operators on $\hs$.  We also consider $2\times 2$ matrices of bounded operators on $L^2(\mu_{\frac{1}{4}})$, which we denote by $M_{2\times 2}(B(L^2(\mu_{\frac{1}{4}})))$.

\begin{definition}\label{Defn:alphabeta}
Given the Hilbert space $L^2(\mu_{\frac{1}{4}})$ and the representation $\{S_0, S_1\}$ of $\mathcal{O}_2$ on $L^2(\mu_{\frac{1}{4}})$, we define the following maps:
\begin{enumerate}[(a)]
\item the map $\alpha_2:B(L^2(\mu_{\frac{1}{4}}))) \rightarrow M_{2\times 2}(B(L^2(\mu_{\frac{1}{4}})))$ by  
\begin{equation}\label{Eqn:alpha}
(\alpha_2(X))_{i,j} = S_i^*XS_j.
\end{equation}
\item the map $\beta_2: M_{2\times 2}(B(L^2(\mu_{\frac{1}{4}})))\rightarrow B(L^2(\mu_{\frac{1}{4}}))$ by 
\begin{equation}\label{Eqn:beta}
\beta_2( (M_{i,j})) = \sum_{i = 0}^1 \sum_{j = 0}^1 S_i M_{i,j} S_j^*.
\end{equation}
\end{enumerate}
\end{definition}

We include the following lemma for completeness. 
\begin{lemma}
The maps $\alpha_2$ and $\beta_2$ in Definition \ref{Defn:alphabeta} are unital, additive, and multiplicative.  In addition, both $\alpha_2$ and $\beta_2$ respect the involutions on $B(L^2(\mu_{\frac{1}{4}}))$ and $M_{2\times 2}(B(L^2(\mu_{\frac{1}{4}})))$. 
\end{lemma}

\begin{proof}
Let $I$ be the identity on $B(L^2(\mu_{\frac{1}{4}}))$; then $(\delta_{i,j}I)_{i,j}$ is the identity on $M_{2\times 2}(B(L^2(\mu_{\frac{1}{4}})))$.  Fix $(i,j)\in \{0,1\}\times\{0,1\}$.  First,
\begin{equation*}
(\alpha_2(I))_{i,j}
= S_i^*S_j 
= (\delta_{i,j}I)_{i,j},
\end{equation*}
so $\alpha_2$ is unital.   Second, given $X,Y \in B(L^2(\mu_{\frac{1}{4}}))$, then
\begin{equation*}
\begin{split}
(\alpha_2(X+Y))_{i,j} 
& = S_i^*(X+Y)S_j
= S_i^*XS_j + S_i^*YS_j \\
& = (\alpha_2(X))_{i,j} + (\alpha_2(Y))_{i,j}.
\end{split}
\end{equation*}
Third, 
\begin{equation*}
\begin{split}
(\alpha_2(X)\alpha_2(Y))_{i,j}
& = \sum_{k=0}^1 (\alpha_2(X))_{i,k} (\alpha_2(Y))_{k.j}\\
& = \sum_{k=0}^1 S_i^*XS_k S_k^*YS_j.
\end{split}
\end{equation*}
Since $i$ and $j$ are fixed, we can move the sum into the middle of the expression and use the Cuntz property that $\sum_{k=0}^1 S_k S_k^* = I$ to see that 
\begin{equation*}
\sum_{k=0}^1 S_i^*XS_k S_k^*YS_j = (\alpha_2(XY))_{i,j}.
\end{equation*}

Finally, if $X\in B(L^2(\mu_{\frac{1}{4}}))$, then
\begin{equation*}
(\alpha_2(X^*))_{i,j}  
= S_i^*X^*S_j 
= (S_j^* X S_i)^* 
= ((\alpha_2(X))_{j,i})^*
\end{equation*}
where the ``$*$'' in $M_{2\times 2}(B(L^2(\mu_{\frac{1}{4}})))$ has the usual matrix meaning.

We now consider $\beta_2$.  The relation $\sum_i\sum_j S_iS_j^* = I$ makes $\beta_2$ a unital map.

Let $\widetilde{X}, \widetilde{Y} \in M_{2\times 2}(B(L^2(\mu_{\frac{1}{4}})))$. By distributing $S_i$ and $S_j^*$ on both sides of
\[ \sum_{i=0}^i \sum_{j=0}^1 S_i(\widetilde{X}_{i,j} + \widetilde{Y}_{i,j})S_j^*,\]
we can see that
\[ \beta_2(\widetilde{X} + \widetilde{Y}) = \beta_2(\widetilde{X}) + \beta_2(\widetilde{Y}).\]
For the product, 
\[
\beta_2(\widetilde{X} \widetilde{Y}) 
=  \sum_{i=0}^1 \sum_{j=0}^1 S_i \Bigl( \sum_{k=0}^1\widetilde{X}_{i,k} \widetilde{Y}_{i,k} \Bigr) S_j^*.
\]
On the other hand,
\[
\beta_2(\widetilde{X})\beta_2 (\widetilde{Y})
 = \sum_{i, j, k, \ell = 0}^1 S_i \widetilde{X}_{i,k} S_k^* S_{\ell} \widetilde{Y}_{\ell, j} S_{j}^*,
\]
and $S_k^* S_{\ell} = \delta_{k, \ell}I$.  Therefore we can remove the $\ell$ index from the sum:
\[
\beta_2(\widetilde{X})\beta_2 (\widetilde{Y})
 = \sum_{i, j, k = 0}^1 S_i \widetilde{X}_{i,k} \widetilde{Y}_{k,j} S_{j}^*
 = \beta_2(\widetilde{X} \widetilde{Y}) .
\]

Finally, if $\widetilde{X} = (\widetilde{X})_{i,j}$, then $\widetilde{X}^* = (\widetilde{X}_{j,i}^*)$.
\[ 
\beta_2(\widetilde{X}^*) 
= \sum_{i=0}^1 \sum_{j=0}^1 S_i \widetilde{X}^*_{j,i} S_j,
\]
and
\[ 
[ \beta_2(\widetilde{X}) ]^*
= \Bigl( \sum_{i=0}^1 \sum_{j=0}^1 S_i \widetilde{X}_{i,j} S_j^*\Bigr)^*
= \sum_{i=0}^1 \sum_{j=0}^1 S_j \widetilde{X}^*_{i,j} S_i^*
.\] 
\end{proof}

\iffalse
Recall that $U$ defined by $Ue_{\gamma} = e_{5\gamma}$ on $E(\Gamma)$ commutes with $S_0$ and $S_0^*$, and therefore with the projections $S_0S_0^*$ and $S_1S_1^* = I - S_0S_0^*$.

By Fuglede's condition, any normal operator $X$ which commutes with $S_0$ also commutes with $S_0^*$.  Since $I = S_0S_0^* + S_1S_1^*$, any normal operator $X$ which commutes with $S_0$ also commutes with the projection $S_1S_1^*$.  To obtain the identities which follow (Equations \eqref{Eqn:b2a2} and \eqref{Eqn:b2a2_2}), we could assume that $X\in B(L^2(\mu_{\frac{1}{4}}))$ is normal and commutes with $S_0$, but the application we have in mind will require non-normal operators which still have the property that they commute with $S_0$ and $S_0^*$.   
\fi

\begin{lemma}\label{Lemma:Id}
Let $\alpha_2$ and $\beta_2$ be the maps defined in Definition \ref{Defn:alphabeta}.
Then 
\begin{equation}\label{Eqn:Id}
\beta_2\circ\alpha_2 = I.
\end{equation}
\end{lemma}
\begin{proof}
Let $X\in B(L^2(\mu_{\frac{1}{4}}))$.  We use the equation $S_0S_0^* + S_1S_1^* = I$ twice:
\begin{equation}\label{Eqn:Id}
\begin{split}
\beta_2(\alpha_2(X)) 
& = S_0 S_0^*X S_0 S_0^* + S_0 S_0^*X S_1 S_1^* \\
& \phantom{==} + S_1 S_1^* X S_0S_0^* +  S_1 S_1^* X S_1 S_1^*\\
& = S_0 S_0^* (X S_0 S_0^* + X S_1 S_1^*)\\
& \phantom{==} + S_1 S_1^* (X S_0 S_0^* + X S_1 S_1^*)\\
& = X S_0 S_0^* + X S_1 S_1^* = X.
\end{split}
\end{equation}
\end{proof}

If we assume that $X$ commutes with $S_0$ and $S_0^*$, then we can derive the following identity from Equation \eqref{Eqn:Id}:
\begin{equation}\label{Eqn:Id1}
\begin{split}
X 
& = S_0 \underbrace{S_0^*X S_0}_{X} S_0^* + S_0 \underbrace{S_0^*X S_1}_{0} S_1^* \\
& \phantom{==} + S_1 \underbrace{S_1^* X S_0}_{0} S_0^* +  S_1 S_1^* X S_1S_1^*\\
& = S_0 X S_0^* +  S_1 S_1^* X S_1S_1^*\\
\end{split}
\end{equation}

\subsection{The projection-valued measure for $U$}\label{Subsec:pvms}
We will now specialize to the case $X = \phi(U)$ where $\phi$ is a Borel function on $\T$ and $U$ is the operator which enacts scaling by $5$ on the canonical spectrum for $L^2(\mu_{\frac{1}{4}})$:
\[ Ue_{\gamma} = e_{5\gamma}.\]
As we saw earlier, $U$ commutes with $S_0$ and therefore with $S_0^*$ by Fuglede's theorem.  Both $U$ and $U^*$ commute with the range projection $S_1S_1^*$.  As a result, $\phi(U)$ commutes with $S_1S_1^*$.

\begin{lemma}\label{Lemma:S_1U^k}Let $U$, $S_0$, $S_1$, and $M_{e_1}$ be defined on $L^2(\mu_{\frac14})$ as above.  Then for each $k\in\N$,
\begin{equation}\label{Eqn:S_1U^k}
S_1^* U^k S_1 = (S_1^* U S_1)^k = (M_{e_1}U)^k
\end{equation}
and
\begin{equation}\label{Eqn:S_1U^*k}
S_1^* U^{-k} S_1 = (S_1^* U S_1)^{-k}.
\end{equation}
\end{lemma}
\begin{proof}
The second equality in Equation \eqref{Eqn:S_1U^k} follows from \cite[Theorem 4.10]{JKS11}.

For the first equality in \eqref{Eqn:S_1U^k}:  when $k = 1$, there is nothing to check.  Suppose the first equality is true for $k = n$.  Then $S_1^* U^n S_1 = (S_1^* U S_1)^n$.  Now compute $S_1^* U^{n+1} S_1$:
\begin{equation}
\begin{split}
S_1^* U^{n+1} S_1
& = S_1^* U^n U S_1
   = S_1^* U^n (S_0S_0^* + S_1 S_1^*)U S_1\\
& = S_1^* U^n S_0S_0^*U S_1 + (S_1^* U^n S_1)(S_1^*U S_1)\\
\end{split}
\end{equation}
Since $U$ and $S_0^*$ commute and $S_0^* S_1 = 0$, the first term in the sum is $0$.  Therefore only the last term remains, and by the induction hypothesis, 
\begin{equation}
(S_1^* U^n S_1)(S_1^*U S_1) = (S_1^* U S_1)^n(S_1^*U S_1) = (S_1^* U S_1)^{n+1}.
\end{equation}

We now turn to Equation \eqref{Eqn:S_1U^*k}.  Because $S_1^*US_1 = M_{e_1}U$, $(S_1^*US_1)^{-1}$ exists and
\[(S_1^*US_1)^{-1} = (S_1^*US_1)^*.\] 
Suppose $k = 1$:
\begin{equation}
(S_1^*US_1)^* = S_1^* U^* (S_1^*)^* = S_1^* U^{-1} S_1.
\end{equation}

For $k > 1$, we use the same induction as before to establish Equation \eqref{Eqn:S_1U^*k}.
\end{proof}

Again, let $\phi:\T\rightarrow \C$ be a Borel function.  If we approximate $\phi(U)$ by finite Laurent series in $U$ and apply  Lemma \ref{Lemma:S_1U^k}, we have
\begin{equation}\label{Eqn:S_1phi}
S_1^* \phi(U) S_1 = \phi(M_{e_1}U).
\end{equation}
No commutation relations are required for \eqref{Eqn:S_1phi}.  Since $\phi(U)$ commutes with $S_0$ and $S_0^*$, Equation \eqref{Eqn:Id1} becomes
\begin{equation}\label{Eqn:Id2}
\phi(U) = S_0 \phi(U) S_0^* + S_1 \phi(M_{e_1}U) S_1^*.
\end{equation}

Let $P^U$ and $P^{M_{e_1}U}$ be the projection-valued measures (pvms) for $U$ and $M_{e_1} U$ respectively.  The pvms are defined by the spectral theorem, which says that
there exists a pvm $P^U$\! such that the unitary operator $U$ can be written as an integral against that pvm: 
\[ U = \int_{\sigma(U)} z\: P^U\!(dz).\]
Here, $\sigma(U)$ is the spectrum of $U$.  See Lemma \ref{Thm:DSX28} for more results associated with the spectral theorem.
 
\begin{proposition}\label{Prop:PVMDecomp}
The projection-valued measure of $U$ has the following decomposition:
\begin{equation}\label{Eqn:MeasDecomp}
P^{U}\!(A) = S_0 P^{U}\!(A) S_0^* + S_1 P^{M_{e_1}U}\!(A) S_1^*.
\end{equation}
\end{proposition}

\begin{proof}
If $A$ is any Borel subset of $\T$, we can set $\phi = \chi_A$.  Then $P^U(A) = \chi_{A}(U)$, so $P^U\!(A)\in B(L^2(\mu_{\frac{1}{4}}))$ commutes with both $S_0$ and $S_0^*$.  Substituting $\phi = \chi_A$ into Equation \eqref{Eqn:Id2} yields Equation \eqref{Eqn:MeasDecomp}.
\end{proof}

\subsection{Decompositions of spectral measures}\label{Subsec:SpecMeas}

We now turn to the scalar measures defined from the pvms $P^{U}$ and $P^{M_{e_1}U}$.  For each vector $v\in L^2(\mu_{\frac14})$ and unitary operator $X$ on $L^2(\mu_{\frac14})$, the real-valued Borel measure $m^X_v$ is defined on the Borel set $A\subset \T$ by
\begin{equation}\label{Eqn:m^{X}_v}
m^{X}_v(A) = \langle P^X\!(A)v, v \rangle_{L^2(\mu_{\frac14})}.
\end{equation}
When $v$ is a unit vector, the measure defined by Equation \eqref{Eqn:m^{X}_v} is a probability measure.
\begin{proposition}\label{Prop:RealDecomp}
Fix a unit vector $v\in L^2(\mu_{\frac{1}{4}})$.  Let $m^{U}_v$ and $m^{MU}_v$
be the spectral measures associated with $U$ and $M_{e_1}U$ respectively.  Then
\begin{equation}\label{Eqn:RealDecomp}
m^{U}_v = m^{U}_{S_0^*v}  + m^{MU}_{S_1^*v}.
\end{equation}
\end{proposition}
\begin{proof}
Let $A\subset \T$ be a Borel set and $v$ a unit vector in $\T$.  Using Equation \eqref{Eqn:MeasDecomp}, we can write
\begin{equation}
\begin{split}
m^{U}_v(A) 
& = \langle P^{U}\!(A)v, v \rangle_{L^2(\mu_{\frac14})}\\
& =  \langle S_0 P^{U}\!(A) S_0^*v + S_1 P^{MU}\!(A) S_1^*v, v \rangle_{L^2(\mu_{\frac14})}\\
& = \langle S_0 P^{U}\!(A) S_0^*v, v\rangle_{L^2(\mu_{\frac14})} +  \langle S_1 P^{MU}\!(A) S_1^*v, v \rangle_{L^2(\mu_{\frac14})}\\
& = \langle P^{U}\!(A) S_0^*v, S_0^*v\rangle_{L^2(\mu_{\frac14})} +  \langle P^{MU}\!(A) S_1^*v, S_1^*v \rangle_{L^2(\mu_{\frac14})}\\
& = m^{U}_{S_0^*v}(A)  + m^{MU}_{S_1^*v}(A). 
\end{split}
\end{equation}
Stated differently,
\begin{equation}\label{Eqn:dm_decomp}
dm^{U}_v(z) = dm^{U}_{S_0^*v}(z)  + dm^{MU}_{S_1^*v}(z)
\end{equation}
\end{proof}

\begin{example}\label{Ex:e_n}The real-valued spectral measures on exponential functions.  
\end{example}
\noindent \textbf{Case 1:  $n$ even.}  Let $n = 2m\in 2\Z$.  Then $e_n\in \overline{E(4\Gamma)}$ since
\[ \widehat{\mu}_{\frac{1}{4}}(4\gamma + 1 -2m) = 0\]
for all $\gamma\in\Gamma$.
Therefore we can choose $h\in L^2(\mu_{\frac{1}{4}})$ such that $e_n = S_0 h$.  As a result,
\begin{equation}
m_{e_n}^U = m^{U}_{S_0^*S_0h}  + m^{MU}_{S_1^*S_0h} = m^{U}_{h},
\end{equation}
since for any Borel set $A\subset \T$,
\[m^{MU}_{e_0}(A) = \langle P^{MU}\!(A)\:e_0, e_0\rangle_{L^2(\mu_{\frac{1}{4}})} = 0\]
and $S_1^*S_0h = 0$.

\noindent\textbf{Case 2:  $n$ odd.  }Let $n = 2m + 1$, and compute
\[ \widehat{\mu}_{\frac{1}{4}}(2m + 1 - 4\gamma) = 0.\]
Therefore $e_n\in \overline{E(4\Gamma + 1)}$ (the closed linear span of $E(4\Gamma + 1)$); in other words, $e_n = S_1 h'$ for some $h'\in L^2(\mu_{\frac{1}{4}})$.  Then
\begin{equation}
m_{e_n}^U = m^{U}_{S_0^*S_1h'}  + m^{MU}_{S_1^*S_1h'} = m^{MU}_{h'}.
\end{equation}
In other words, $m_{S_1h'}^U = m^{MU}_{h'}$ for the $h'$ we chose above.  
\hfill$\Diamond$

\begin{example}\label{Ex:S_0h}Let $h\in L^2(\mu_{\frac{1}{4}})$ be a unit vector.  Then for each $k\in\N$, $m^U_{h} = m^U_{S_0^k h}$.
\end{example}
\noindent  Since $S_0$ is an isometry, $S_0^nh$ is a unit vector for all $n\in\N$.  For $k = 1$, we perform almost the same calculation as before:  for any unit vector $h$,
\[m_{S_0h}^U = m^{U}_{S_0^*S_0h}  + m^{MU}_{S_1^*S_0h} = m^{U}_{h}.\]
Now suppose $m^U_{v} = m^U_{S_0^n v}$ for all unit vectors $v$ and for $k =n$.  Then, regrouping, we have
$m_{S_0^{n+1}h}^U 
 = m_{S_0^n S_0h}$, which by the inductive hypothesis for $v = S_0h$, is $ m_{S_0h}$.  
Then by our base case,
$m_{S_0h} 
= m_h$.
\hfill$\Diamond$

\begin{corollary}\label{Cor:EqualMeasures}
For each $k\in \N$, $m^U_{e_1} = m^U_{e_{4^k}}$.
\end{corollary}
\begin{proof}
We apply Example \ref{Ex:S_0h} and induction: $e_{4^k} = S_0^ke_1$.
\end{proof}

Even though determining $m^{U}_h$ for a general unit vector $h$ can be difficult, the following example shows how some cross-terms can be eliminated fairly easily with Proposition \ref{Prop:RealDecomp}.

\begin{example}\label{Ex:CrossTerms}
One more simple example.
\end{example}

\noindent Let $h = \frac{1}{\sqrt{2}}(e_0 + e_5)$, and let $A\subset \T$ be a Borel set.  Then without Proposition \ref{Prop:RealDecomp} which allows for quick elimination of cross-terms, the computation of $m^U_h$ would be bulky:\begin{equation}
\begin{split}
m^U_h & = \Biggl\langle P^U\!(A)\Biggl(\frac{1}{\sqrt{2}}(e_0 + e_5)\Biggr), \frac{1}{\sqrt{2}}(e_0 + e_5)\Biggr\rangle\\
& = \frac{1}{2} \langle P^U\!(A)(e_0 + e_5), (e_0 + e_5)\rangle\\
& = \frac{1}{2} m^U_{e_0} + \frac{1}{2} m^{MU}_{e_5}\\
& \phantom{==}  + \frac{1}{2} \langle P^U\!(A)e_0, e_5\rangle + \frac{1}{2} \langle P^U\!(A)e_5, e_0\rangle.
\end{split}
\end{equation}

On the other hand, writing $h = \frac{1}{\sqrt{2}}(S_0e_0 + S_1e_1)$ makes the cross-terms disappear quickly in the
application of Equation \eqref{Eqn:RealDecomp}:
\begin{equation}
\begin{split}
m^{U}_h 
& = m^{U}_{S_0^*(\frac{1}{\sqrt{2}}(S_0e_0 + S_1e_1))}  + m^{MU}_{S_1^*(\frac{1}{\sqrt{2}}(S_0e_0 + S_1e_1))}\\
& = m^U_{\frac{1}{\sqrt{2}}e_0} + m^{MU}_{\frac{1}{\sqrt{2}}e_1}\\
\end{split}
\end{equation}
The measure $m^{U}_{e_0}$ is the Dirac mass $\delta_1$ because $Ue_0 = e_0$.  A proof is given in \cite[Proof of Lemma 4.1, $(\Leftarrow)$ direction]{JKS12}.
\hfill$\Diamond$

\begin{theorem}\label{Thm:Iterate}
Let $v\in L^2(\mu_{\frac14})$.  Then
\begin{equation}\label{Eqn:Iterate_2}
m_v^U = |\langle v, e_0\rangle|^2\delta_1 + \sum_{k=0}^{\infty}m_{S_1^*S_0^{*k}v}^{MU}.
\end{equation}
\end{theorem}
\begin{proof}
By Equation \eqref{Eqn:dm_decomp} we have $m_v^U = m_{S_0^*v}^U + m_{S_1^*v}^{MU}$.  Apply \eqref{Eqn:dm_decomp} to $S_0^*v$ to obtain
\begin{equation}
\begin{split}
m_v^U 
& = m_{S_0^*v}^U + m_{S_1^*v}^{MU}\\
& = m_{S_0^{*2}v}^U + m_{S_1^*S_0^*v}^{MU} + m_{S_1^*v}^{MU}.\\
\end{split}
\end{equation}
Continuing the process, we find that for any $n\in\N$,
\begin{equation}\label{Eqn:Iterate}
m_v^U = m_{S_0^{*n}v}^U + \sum_{k=0}^{n-1}m_{S_1^*S_0^{*k}v}^{MU}.
\end{equation}
We claim that
\begin{equation}\label{Eqn:Claim}
\lim_{n\rightarrow\infty} m_{S_0^{*n}v}^U = m^U_{P_{e_0}v} = |\langle v, e_0\rangle|^2 \delta_1.
\end{equation}
The right-hand side of Equation \eqref{Eqn:Claim} follows from the definition of $m^U_{P_{e_0}v}$.
Let $\phi\in C(\T)$.  Then
\begin{equation}
\begin{split}
&  \int_{\T} \phi(z) \: dm^U_{P_{e_0}v} 
= \langle P_{e_0}v, \phi(U)P_{e_0}v\rangle_{L^2(\mu_{\frac14})}\\
& = \Bigl\langle \langle v, e_0\rangle e_0, \phi(U)\langle v, e_0\rangle e_0\Bigr\rangle_{L^2(\mu_{\frac14})}\\
& = |\langle v, e_0\rangle|^2 \langle e_0, \phi(U)e_0\rangle_{L^2(\mu_{\frac14})}\\
& = |\langle v, e_0\rangle|^2  \int_{\T} \phi(z) \: dm^U_{e_0} = |\langle v, e_0\rangle|^2 \phi(1),\\
\end{split}
\end{equation}
by Example \ref{Ex:CrossTerms}.

To establish the left-hand side of Equation \eqref{Eqn:Claim}, we show
\begin{equation}\label{Eqn:ToEstimate}
\lim_{n\rightarrow\infty}\Big| \langle S_0^{*n}v, \phi(U) S_0^{*n}v\rangle_{L^2(\mu_{\frac14})} - |\langle v, e_0\rangle|^2 \phi(1)\Big| = 0.
\end{equation}
for any $\phi\in C(\T)$.
Since $S_0$ is an isometry and $S_0\in \{U\}'$, we can perform the following operations:
\begin{equation}
\begin{split}
& \Big| \langle S_0^{*n}v, \phi(U) S_0^{*n}v\rangle_{L^2(\mu_{\frac14})} - |\langle v, e_0\rangle|^2 \phi(1)\Big|\\
& =  \Big|  \langle S_0^nS_0^{*n}v, \phi(U) S_0^nS_0^{*n}v\rangle_{L^2(\mu_{\frac14})} - |\langle v, e_0\rangle|^2 \phi(1)\Big|\\
\end{split}
\end{equation}
For ease of notation, let $P_n$ denote $S_0^nS_0^{*n}$.  By properties of the Wold decomposition, $\{P_n\}$ is a decreasing sequence of projections, and by Proposition \ref{Prop:WoldS_0}, part (a), we have 
\begin{equation}\label{Eqn:WoldLimit}
\lim_{n\rightarrow\infty} P_n v = P_{e_0}v.
\end{equation}
Now, back to Equation \eqref{Eqn:ToEstimate}:
\begin{equation*}
\begin{split}
& \Big| \langle P_n v, \phi(U) P_n v\rangle_{L^2(\mu_{\frac14})} - |\langle v, e_0\rangle|^2 \phi(1)\Big|\\
& = \Big| \langle (P_n - P_{e_o}) v, \phi(U) P_n v\rangle_{L^2(\mu_{\frac14})} + \langle P_{e_0} v, \phi(U) (P_n-P_{e_0}) v\rangle_{L^2(\mu_{\frac14})}\\
& \phantom{==} + \langle P_{e_0} v, \phi(U) P_{e_0}v\rangle _{L^2(\mu_{\frac14})} - |\langle v, e_0\rangle|^2 \phi(1)\Big|\\
& =  \Big| \langle (P_n - P_{e_o}) v, \phi(U) P_n v\rangle_{L^2(\mu_{\frac14})} + \langle P_{e_0} v, \phi(U) (P_n-P_{e_0}) v\rangle_{L^2(\mu_{\frac14})}\Big|.\\
\end{split}
\end{equation*}
We can now apply Cauchy-Schwarz to the two remaining inner products.  We have
\begin{equation*}
\begin{split}
& \Big|\langle (P_n - P_{e_o}) v, \phi(U) P_n v\rangle_{L^2(\mu_{\frac14})}\Big|\\
& \leq \|(P_n - P_{e_0})v\| \underbrace{\|\phi(U)\|_{\textrm{op}} \|v\|}_{\textrm{constant}}\\
& \rightarrow 0 \textrm{ as }n\rightarrow\infty \textrm{ by } \eqref{Eqn:WoldLimit}
\end{split}
\end{equation*}
and
\begin{equation*}
\begin{split}
& \Big|\langle P_{e_0} v, \phi(U) (P_n-P_{e_0}) v\rangle_{L^2(\mu_{\frac14})}\Big|\\
& \leq \underbrace{|\langle e_0, v\rangle| \|\phi(U)\|_{\textrm{op}}}_{\textrm{constant}} \|(P_n-P_{e_0}) v\|\\
& \rightarrow 0 \textrm{ as }n\rightarrow\infty \textrm{ also by } \eqref{Eqn:WoldLimit}.
\end{split}
\end{equation*}
Therefore Equation \eqref{Eqn:ToEstimate} is true, and Equation \eqref{Eqn:Iterate} then becomes
\begin{equation}
\begin{split}
m_v^U = |\langle v, e_0\rangle|^2\delta_1  + \sum_{k=0}^{\infty}m_{S_1^*S_0^{*k}v}^{MU}.
\end{split}
\end{equation}
\end{proof}

\noindent\textbf{Remark 1:  }Normalization in the infinite expansion in Theorem \ref{Thm:Iterate}.

We can write Equation \eqref{Eqn:Iterate_2} in terms of probability measures denoted by $\widetilde{m}_{S_1^*S_0^{*k}v}^{MU}$ (the tilde denotes that we start with a unit vector, so that $S_1^*S_0^{*k}v$ is normalized:
\begin{equation}\label{Eqn:Normalized}
m_v^U = |\langle v, e_0\rangle|^2\delta_1 + \sum_{k=0}^{\infty}\| S_1^*S_0^{*k}v \|^2\widetilde{m}_{S_1^*S_0^{*k}v}^{MU},
\end{equation}
where $\|S_1^* S_0^{*k}\|^2 = P^v_{\Gamma}(A(\underbrace{0, \ldots, 0}_{k}, 1))$ is the measure discussed in Section \ref{Sec:Induced}; see \eqref{Eqn:PGamma} and Lemma \ref{Lemma:Borel}.
Note that the constants in \eqref{Eqn:Normalized} are independent of the operator $U$ which denotes scaling by $5$.  Therefore, if another $U$ satisfies the conditions set out in Section \ref{Sec:GenRel} and $v$ is a unit vector, then $m^U_v$ has the same type of decomposition as in \eqref{Eqn:Iterate_2}.

\bigskip

We can make more quantitative observations about Equation \eqref{Eqn:Normalized}.  Let $v$ be a unit vector in $L^2(\mu_{\frac14})$, where 
\begin{equation}
v = \sum_{\gamma\in\Gamma(\frac14)} \langle e_{\gamma}, v\rangle e_{\gamma}  = \sum_{\gamma\in\Gamma(\frac14)} c_v(\gamma) e_\gamma.
\end{equation}
From this expansion, we have
\begin{equation}
\|v\|^2 = \sum_{\gamma\in\Gamma(\frac14)} |c_v(\gamma)|^2.
\end{equation}
Now consider the mutually disjoint subsets of $\Gamma(\frac14)$ from Equation \eqref{Eqn:Decomp} and strings of $0$s and $1$s from Equation \eqref{Eqn:Gamma}:
\begin{itemize}
\item $1 + 4\Gamma(\frac14)$ corresponds to strings in $0$ and $1$ of the form $(1, *, *, \ldots)$ and to the subspace $S_1L^2(\mu_{\frac14})$
\item $4(1 + 4\Gamma(\frac14))$ corresponds to strings in $0$ and $1$ of the form $(0, 1, *, *, \ldots)$ and to the subspace
$S_0S_1L^2(\mu_{\frac14})$
\item and in general, $4^k(1+4\Gamma(\frac14))$ corresponds to strings in $0$ and $1$ of the form $(\underbrace{0, \ldots, 0}_k,1,*,*,\cdots)$ and to the subspace $S_0^kS_1L^2(\mu_{\frac14})$.
\end{itemize}
We can then interpret the normalization constant $\|S_1^*S_0^{*k}v\|^2$ as the probability assigned to all infinite words in the third bullet above---that is,
\begin{equation}\label{Eqn:ProbNorm}
\|S_1^*S_0^{*k}v\|^2 = \sum_{\gamma\in 4^k(1+4\Gamma)} |c_v(\gamma)|^2
\end{equation}
is the probability assigned to all infinite words beginning with the string $(\underbrace{0, \ldots, 0}_{k}1)$.

Another interpretation of the normalization constants in Equation \eqref{Eqn:Normalized} is the following.  Let $P_k$ be the projection onto the subspace $S_0^kS_1L^2(\mu_{\frac14})$---that is,
\begin{equation}\label{Eqn:Pk}
P_k = (S_0^kS_1) (S_0^kS_1)^*.
\end{equation}
Note this is not the same $P_k$ associated with the Wold decomposition in Theorem \ref{Thm:Iterate}.  Then
\begin{equation}
\|S_0^{*k}S_1^*v\|^2 = \|P_k v\|^2;
\end{equation}
therefore the normalized measure in \eqref{Eqn:Normalized} can be written
\begin{equation}
m_{S_0^{*k}S_1^*v}^{MU} =   \| P_k v \|^2\widetilde{m}_{S_1^*S_0^{*k}v}^{MU}.
\end{equation}
\begin{corollary}\label{Cor:Convex}
Let $v\in L^2(\mu_{\frac14})$ with $\|v\| = 1$.  Let $\rm{Prob}_{1}(\T)$ denote the space of probability measures on the circle $\T$.  Then every measure of the form $m_v^U\in \rm{Prob}_{1}(\T)$ has a convex representation
\begin{equation}\label{Eqn:Convex}
m_v^U = |\langle e_0, v\rangle|^2\delta_1 + \sum_{k=0}^{\infty} \| P_k v \|^2\widetilde{m}_{S_1^*S_0^{*k}v}^{MU},
\end{equation}
where $\delta_1$ and $\widetilde{m}_{S_1^*S_0^{*k}v}^{MU}$ belong to $\rm{Prob}_{1}(\T)$.
\end{corollary}
\begin{proof}
Recall that
\begin{equation}
\sum_{k=0}^{\infty} P_k = I_{L^2(\mu_{\frac14})} - P_{e_0}.
\end{equation}
Then
\begin{equation}
1 = \|v\|^2 =  |\langle e_0, v\rangle|^2 + \sum_{k=0}^{\infty}  \| P_k v \|^2,
\end{equation}
which shows that \eqref{Eqn:Convex} is indeed a convex expansion.
\end{proof}

We can connect the normalization constants above to the induced measures studied in Section \ref{Sec:Induced}.  Let
\[
\gamma\leftrightarrow (a_0, a_1, \ldots, a_N) \textrm{ and }\xi \leftrightarrow (b_0, b_1, \ldots, b_K).
\]  
Then the concatenation $\gamma\xi$ represents 
\[ \sum_{i = 0}^N a_i 4^i + \sum_{i = 0}^K b_i 4^{1+N+i}\in \Gamma\Bigl(\frac14\Bigr).\]

Next, consider the Cuntz operators $S_0$ and $S_1$.  If $\eta \leftrightarrow (c_0, \ldots, c_M)$ is a finite word in $0$s and $1$s, then let
\[ S_{\eta}: = S_{c_0} \cdots S_{c_M}.\]
There is a straightforward action of $S_0^*$ and $S_1^*$ on $e_{\gamma\xi}$.  We have
\begin{equation}\label{Eqn:Erase}
S_{\gamma'}^*e_{\gamma\xi} = \begin{cases}e_{\xi} & \gamma = \gamma'\\
                                                                                0  & \gamma \neq \gamma'.
                                                                                \end{cases}
\end{equation}

Recall the projection $P_k$ defined in \eqref{Eqn:Pk}.  Let
\[ 
v = \sum_{\gamma\in \Gamma(\frac14)} c_v(\gamma) e_{\gamma}.
\]
Then
\begin{equation}
\|P_k v\|^2  = \Big\| \sum_{\gamma\in \Gamma(\frac14)} c_v(\gamma) S_1^*S_0^{*k} e_{\gamma} \Big\|^2
\end{equation}
By Equation \eqref{Eqn:Erase} (or by what we know about projections), the only terms $e_{\gamma}$ to survive are $\gamma\in 4^k(1+4\Gamma(\frac14))$.
Therefore we can write
\begin{equation}
\begin{split}
\|P_k v\|^2 
& =  \sum_{\gamma\in 4^k(1+\Gamma(\frac14))} |c_v(\gamma) |^2\\
& = \sum_{\xi\in \Gamma(\frac14)}  |c_v(\underbrace{0, 0, \ldots 0}_{k}, 1 \, \xi) |^2,\\
\end{split}
\end{equation}
which, by Equation \eqref{Eqn:PGamma} is the probability of the cylinder set
\begin{equation}\label{Eqn:PGamma_2}
P^v_{\Gamma}(A(\underbrace{0, 0, \ldots 0}_{k}, 1))
\end{equation}
for the induced measures in Section \ref{Sec:Induced}.

\noindent \textbf{Remark 2:  }In Theorem \ref{Thm:Iterate} below, we show that $m^U_v$  has a representation in terms of bi-measures, i.e., measures in two variables.  On the right-hand side in \eqref{Eqn:Iterate_2}, we have a bi-measure, where the first variable in this bi-measure is the $\sigma$-algebra of Borel sets in $\T$ and the second variable is Borel sets in $X$, the Cantor group in Section \ref{Sec:Induced}.  In the second variable, the bi-measure on the right-hand side in \eqref{Eqn:Iterate_2} is evaluated on ``tail-events''--- i.e., on the cylinder sets  $A(\underbrace{0, \ldots, 0}_{k}, 1)$ (in other words, the cylinder-set of the word beginning with $k$ $0$s, followed by a single $1$.)

\bigskip
\section{Radon-Nikodym derivatives and cyclic subspaces}\label{Sec:Cyclic}  

We now return to the unitary scaling operator $U$. We find some key spectral properties for $U$; we prove, among other things, that the operator has a number of intrinsic fractal features. We will use families of Radon-Nikodym derivatives for our purpose.  We study operators in the commutant of $U$, and we focus on cyclic subspaces within $L^2(\mu_{\frac{1}{4}})$.  

\begin{definition}
The \textbf{cyclic subspace} $\langle v \rangle_{U}\subset L^2(\mu_{\frac{1}{4}})$ is the closed span of the set $\{ U^kv \:|\: k\in\Z\}$.
\end{definition}
By \cite[Lemma 2.5]{JKS12}, the cyclic subspace generated by $v$ with respect to $U$ is
\[ \langle v \rangle_{U} = \{ \phi(U)v \:|\: \phi \in L^2(m_v^U)\}.\]  

\begin{lemma}\label{Lemma:CommuteAC}
Suppose $S$ commutes with $U$ and $U^*$ (that is, $S$ belongs to the commutant of $U$).  Then
\begin{equation}
m_{Sv}^{U} \ll m_v^{U}.
\end{equation}
\end{lemma}
\begin{proof}
Let $A\subset \T$ be a Borel set.  Then
\begin{equation*}
m_{Sv}^{U}(A)  
= \langle P^{U}\!(A) Sv, Sv\rangle_{L^2(\mu_{\frac14})}
= \langle P^{U}\!(A) Sv, P^{U}\!(A)Sv\rangle_{L^2(\mu_{\frac14})}.
\end{equation*}
Since $P^{U}\!(A) = \chi_{A}(U)$ is a function of $U$, $S$ commutes with $P^{U}\!(A)$.  Therefore
\begin{equation*}
\begin{split}
m_{Sv}^{U}(A)  
& = \langle S P^{U}\!(A)v, S P^{U}\!(A)v\rangle_{L^2(\mu_{\frac14})}\\
& = \| S P^{U}\!(A)v\|^2_{L^2(\mu_{\frac14})}\\
& \leq \|S\|^2_{\textrm{op}}  \| P^{U}\!(A)v\|^2_{L^2(\mu_{\frac14})}\\
& = \|S\|^2_{\textrm{op}} m^{U}_v(A).
\end{split}
\end{equation*}
If $m^{U}_v(A) = 0$, then $m_{Sv}^{U}(A) =0$ as well.
\end{proof}

\noindent \textbf{Remark:  }By \cite[Theorem 3, p.~83]{Nel69}, if $h$ belongs to the $U$-cyclic subspace generated by $v$, then $m_h^{U} \ll m_v^{U}$.  Nelson's cyclic subspaces are generated by self-adjoint operators, but his arguments can be changed for $U$,$U^*$-invariant subspaces.

\bigskip

As a result of Lemma \ref{Lemma:CommuteAC}, if $S$ belongs to the commutant of $U$, then the Radon-Nikodym derivative 
\[ \frac{dm_{Sv}^{U}}{dm_v^{U}}\]
exists.  In fact, we can use this Radon-Nikodym derivative in connection with $U$-cyclic subspaces of $L^2(\mu_{\frac14})$.

The following lemma reminds the reader of the consequences of the Spectral Theorem (it appears as Lemma 2.3 in \cite{JKS12}, which is taken from  \cite[Chapter X.2]{DunSch63}, Corollaries X.2.8 and X.2.9).

\begin{lemma}\label{Thm:DSX28}
Suppose $U$ is a unitary operator on the Hilbert space $\hs$ with associated p.v.m.~ $P^U$\!, so that 
\[ U = \int_{\sigma(U)} z\: P^U\!(dz).\]  
Suppose $\phi, \phi_1, \phi_2:\T\rightarrow\C$ are $P^U\!$-essentially bounded, Borel-measurable functions.  Define 
\begin{equation}\label{Eqn:FunCal}
\pi_{U}(\phi) = \phi(U) = \int_{\sigma(U)} \phi(z) \,P^U\!(dz).
\end{equation}
Then
\begin{enumerate}[(i)]
\item $[\phi(U)]^* = \overline{\phi}(U)$.  In other words, $\pi_{U}$ is a $*$-homomorphism.
\item $\pi_{U}(\phi_1\phi_2) = \pi_U(\phi_1)\pi_U(\phi_2)$, and as a result, the operators $\phi_1(U)$ and $\phi_2(U)$ commute.
\item If $\phi(z) \equiv 1$, then $\phi(U)$ is the identity operator.
\item The operator $\phi(U)$ is bounded.
\end{enumerate}
In addition, the operator $\phi(U)$ is bounded if and only if $\phi$ is $P^U\!$-essentially bounded.
\end{lemma}

The result in Theorem \ref{Thm:NoProj} would be less surprising if $S$ were normal, but in our applications, $S$ will be an isometry.  We work within $C(\T)$ because of the restriction in Lemma \ref{Thm:DSX28}.  However, Lemma \ref{Thm:DSX28} does not impose too great a restriction, since $\{ \pi_U(\psi)v:\psi\in C(\T)\}$ is dense in $\langle v\rangle_U$.

\begin{theorem}\label{Thm:NoProj}
Fix a unit vector $v\in L^2(\mu_{\frac14})$. 
Suppose $S \in B(L^2(\mu_{\frac14}))$ commutes with both $U$ and $U^*$, and suppose $\langle v\rangle_{U}$  is invariant under $S$.  Then for any $\psi\in C(\T)$, 
\begin{equation}\label{Eqn:NoProj}
S\pi_U(\psi)v 
%= \sqrt{\frac{dm_{Sv}}{dm_v}}(U)w 
= \pi_U\Bigl( \sqrt{\frac{dm_{Sv}}{dm_v}} \psi \Bigr)v.
\end{equation}
\end{theorem}

\noindent\textbf{Remark:  }Another way to state the conclusion of Theorem \ref{Thm:NoProj} is to say that for every $w\in \{ \pi_U(\psi)v:\psi\in C(\T)\}$,
\[Sw = \pi_U\Bigl( \sqrt{\frac{dm_{Sv}}{dm_v}} \psi\Bigr)w,\]
so that the operator $S$ is the operator $\pi_U\Bigl( \sqrt{\frac{dm_{Sv}}{dm_v}}\psi\Bigr)$ on all such $w$.

\begin{proof}
Let $\psi\in C(\T)$ be arbitrary.  Then $\psi(U)\in B(L^2(\mu_{\frac14}))$ because $\psi$ is a bounded function.  Since $S$ commutes with $U$ and $U^*$,
\[ S\pi_U(\psi)v = \pi_U(\psi) Sv.\]
  Since $Sv\in \langle v\rangle_U$ by hypothesis, we can use Theorem 3.6 in \cite{JKS12} to write
\[ Sv = \pi_U\Bigl( \sqrt{\frac{dm_{Sv}}{dm_v}} \psi\Bigr)v.\]
Therefore 
\[\pi_U(\psi) Sv = \pi_U(\psi)\pi_U\Bigl( \sqrt{\frac{dm_{Sv}}{dm_v}} \Bigr)v.\]
If $\pi_U\Bigl( \sqrt{\frac{dm_{Sv}}{dm_v}} \Bigr)$ is bounded, then \textit{(ii)} in Lemma \ref{Thm:DSX28} automatically applies, and we can commute the two operators:
\[ \pi_U(\psi)\pi_U\Bigl( \sqrt{\frac{dm_{Sv}}{dm_v}} \Bigr)v =\pi_U\Bigl( \sqrt{\frac{dm_{Sv}}{dm_v}} \Bigr)\pi_U(\psi)v. \]

If $\pi_U\Bigl( \sqrt{\frac{dm_{Sv}}{dm_v}} \Bigr) \in L^2(m_v)$ is not bounded, then we can still switch the order of the operators  $\pi_U(\psi)$ and $\pi_U\Bigl( \sqrt{\frac{dm_{Sv}}{dm_v}} \Bigr)$, since we can approximate $\sqrt{\frac{dm_{Sv}}{dm_v}} \in L^2(m_v)$ with bounded  functions $\{f_n\}$, and $\pi_U(f_n)\pi_U(\psi) =\pi_U(\psi) \pi_U(f_n)$ for each $n$.
\end{proof} 

When we drop the hypothesis that $S\langle v\rangle_U \subseteq \langle v\rangle_U$, we cannot substantially improve the result in Theorem \ref{Thm:NoProj}. 

\begin{example}
\label{Ex:Project} The Radon-Nikodym derivative in Theorem \ref{Thm:NoProj} must contain a projection if $\langle v\rangle_U$ is not invariant under $S$.
\end{example}
\noindent If the representation of $U$ did not have multiplicity, we would be able to use the Nelson isomorphism \cite[Theorem 4, p.~86]{Nel69} to conclude that $m_{P^{\perp}Sv}$ is mutually singular with respect to $m_{Sv}$, yielding the following simplified equation:
\begin{equation}\label{Eqn:PLHS}
PS\pi_U(\psi)w = \pi_{U}\Bigl(\sqrt{\frac{dm_{Sv}^{U}}{dm_v^{U}}}\psi\Bigr)w.
\end{equation}
However, a simple example in our familiar setting $L^2(\mu_{\frac{1}{4}})$ shows that this equation is not true.   One can decompose $L^2(\mu_{\frac{1}{4}})$ into $U$-cyclic subspaces in which $\langle e_1\rangle_U$ has infinite multiplicity.

Set $v = e_1$, $S = S_0$,  and $\psi(z) = 1$ for all $z\in\T$.  Choose $w = \phi(U)e_1$ where $\phi(z) = 1$ for all $z\in\T$, so that $w = e_1$ as well.  Then the left-hand side of Equation \eqref{Eqn:PLHS} is 
\[ PS_0\phi(U)w = P S_0 w = P S_0 e_1 = P e_4 = 0,\]
since $e_4\in S_0S_1\overline{E(\Gamma(\frac{1}{4}))}$, $e_1 \in S_1\overline{E(\Gamma(\frac{1}{4}))}$, and the two spaces are orthogonal.  On the right-hand side, the Radon-Nikodym derivative
\[ \frac{dm^U_{S_0e_1}}{dm^U_{e_1}}=\frac{dm^U_{e_4}}{dm^U_{e_1}} = 1\]
since $m_{e_1} = m_{e_4}$ by Corollary \ref{Cor:EqualMeasures}.  Therefore the right-hand side is $e_1\neq 0$.  \hfill$\Diamond$

In general, when $\langle v\rangle_U$ is not invariant under $S$, one can replace $S$ with $PS$, where $P$ is the orthgonal projection onto $\langle v\rangle_U$.  Then Equation \eqref{Eqn:NoProj} becomes
\begin{equation}\label{Eqn:Proj}
PS\pi_U(\psi)v 
= \pi_U\Bigl( \sqrt{\frac{dm_{PSv}}{dm_v}} \psi \Bigr)v.
\end{equation}

\begin{example}Radon-Nikodym derivatives in Theorem \ref{Thm:Iterate} and a transitive $U$-action.\end{example}
\noindent From Equation \eqref{Eqn:Iterate_2} in Theorem \ref{Thm:Iterate}, we know that
\begin{equation}
m^{MU}_{S_1^*S_0^{*k}v} \ll m^U_v,
\end{equation}
so we can compute the Radon-Nikodym derivative.  Let 
\begin{equation}
|F^{(v)}_k|^2 = \frac{dm^{MU}_{S_1^*S_0^{*k}v}}{dm^U_v},
\end{equation}
where $F^{(v)}_k$ depends on both $k$ and $v$. 

We examine the consequences of the equation 
\begin{equation}
\int_{\T} \phi(z) \:dm^{MU}_{S_1^*S_0^{*k}v} = \int_{\T} \phi(z)|F^{(v)}_k(z)|^2\:dm^U_v,
\end{equation}
where $\phi$ runs through all the continuous functions on $\T$.  Following exactly the same inductive reasoning in Lemma \ref{Lemma:S_1U^k} and its following discussion, we can establish that 
\begin{equation}
\int_{\T} \phi(z) \:dm^{MU}_{S_1^*S_0^{*k}v} 
= \langle S_0^k S_1S_1^* S_0^{*k}v, \phi(U)v\rangle_{L^2(\mu_{\frac14})}.
\end{equation}
Here, we use the fact that $U$ commutes with $S_1S_1^*$ and that the finite Laurent series 
\[\sum_{k = -N}^N a_k[S_1S_1^* U]^k\]
approximate $\phi(U)$.
On the other hand,
\begin{equation}
\int_{\T} \phi(z)|F^{(v)}_k(z)|^2\:dm^U_v 
= \langle |F^{(v)}_k(U)|^2v, \phi(U)v\rangle_{L^2(\mu_{\frac14})}.
\end{equation}
Therefore
\begin{equation}\label{Eqn:RN_noU}
\langle S_0^k S_1S_1^* S_0^{*k}v, \phi(U)v\rangle_{L^2(\mu_{\frac14})}
=\langle |F^{(v)}_k(U)|^2v, \phi(U)v\rangle_{L^2(\mu_{\frac14})}.
\end{equation}
Notice again that the left-hand side of \eqref{Eqn:RN_noU} does not depend on $U$, except that we use $\phi(U)$ to move $v$ around the $U$-invariant subspace generated by $v$.\hfill$\Diamond$

For each $k$, the left-hand side in \eqref{Eqn:RN_noU} refers to the projection $P_k$  of $v$  onto the range of $S_0^k S_1$  where the chosen vector  $v$ belongs to $L^2(\mu_{\frac14})$. 
Since $\{P_k\}$  is an orthogonal family of projections adding up to $L^2(\mu_{\frac14})$ except for the one-dimensional subspace in $L^2(\mu_{\frac14})$ spanned by the constant function, we say that $\{P_k\}$ is a \textbf{transitive} family of projections.  

Now the left-hand side in \eqref{Eqn:RN_noU} shows that the contribution to the scalar measure from $\{P_k  v \}$  is accounted for by functions of $U$ applied to $v$. Hence the action of $U$ and of functions of $U$ is transitive on  $L^2(\mu_{\frac14})$ in this sense.

%\end{doublespace}

\section{Acknowledgements}The authors would like to thank Christopher French of Grinnell College for many illuminating conversations about this work.

\bibliographystyle{alpha}
\def\lfhook#1{\setbox0=\hbox{#1}{\ooalign{\hidewidth
  \lower1.5ex\hbox{'}\hidewidth\crcr\unhbox0}}}

%\bibliography{U_paper_Karen}

\end{document}